\documentclass[pdflatex,sn-mathphys-num]{sn-jnl}

\usepackage{amssymb,amsmath,amscd,amsthm,pb-diagram,tikz-cd,url}
\usepackage{graphicx}

\usepackage{hyperref}

\newcommand{\bfg}{{\mathbf g}}
\newcommand{\bfh}{{\mathbf h}}

\newcommand{\bfp}{{\mathbf p}}
\newcommand{\bfs}{{\mathbf s}}
\newcommand{\bfB}{{\mathbf B}}
\newcommand{\bfC}{{\mathbf C}}
\newcommand{\bfI}{{\mathbf I}}
\newcommand{\bfM}{{\mathbf M}}
\newcommand{\bfP}{{\mathbf P}}

\newcommand{\bfS}{{\mathbf S}}
\newcommand{\bfT}{{\mathbf T}}
\newcommand{\bfa}{{\mathbf a}}
\newcommand{\bfb}{{\mathbf b}}
\newcommand{\bfe}{{\mathbf e}}
\newcommand{\bff}{{\mathbf f}}
\newcommand{\bfr}{{\mathbf r}}
\newcommand{\bfu}{{\mathbf u}}
\newcommand{\bfv}{{\mathbf v}}
\newcommand{\bfw}{{\mathbf w}}
\newcommand{\bfy}{{\mathbf y}}

\newcommand{\R}{{\mathbb R}}
\newcommand{\Z}{{\mathbb Z}}
\newcommand{\ol}[1]{{\overline{#1}}}
\newcommand{\wh}[1]{{\widehat{#1}}}
\DeclareMathOperator{\Hom}{{Hom}}
\DeclareMathOperator{\id}{Id}

\newtheorem{theorem}{Theorem}

\newtheorem{conjecture}[theorem]{Conjecture}
\newtheorem{corollary}[theorem]{Corollary}
\newtheorem{lemma}[theorem]{Lemma}
\theoremstyle{definition}
\newtheorem{definition}[theorem]{Definition}
\newtheorem{remark}[theorem]{Remark}
\newtheorem{example}[theorem]{Example}

\begin{document}

\title{Characteristic tensors for almost Finsler manifolds}

\author{\fnm{James F.}\sur{Davis$^{1,2}$}}

\author{\fnm{Benjamin R.}\sur{Edwards$^2$}}

\author{\fnm{V.\ Alan}\sur{Kosteleck\'y$^{1,2}$}}

\affil{\hfil$^1$Department of Mathematics and Department of Physics\hfil \\ 
Indiana University, Bloomington, IN 47405 USA}

\affil{\hfil$^2$Indiana University Center for Spacetime Symmetries\hfil \\ 
Indiana University, Bloomington, IN 47405 USA}

\abstract{
Almost Finsler manifolds and partial Finsler manifolds are introduced,
extending the standard definition of a Finsler manifold
to allow for a nontrivial slit 
containing points fixed under homogeneous scaling
and for metrics where the fundamental tensor has nonpositive eigenvalues.
The bipartite spaces offer examples 
of comparatively simple almost Finsler manifolds 
and partial Finsler manifolds with physics applications.
Special cases are the $\bfa$ and $\bfb$ spaces,
which have almost Finsler norms and partial Finsler norms 
formed from a Riemannian norm and a 1-form.
The indicatrix union of the almost Finsler $\bfa$ manifolds 
equals the indicatrix union of Randers spaces.
Characteristic tensors that vanish for bipartite spaces and $\bfb$ spaces
are obtained and expressed using geometric quantities.
These tensors are generalizations of the Matsumoto tensor,
which vanishes on Randers and $\bfa$ spaces.

\vskip 30pt
\hfil
{\large 
Published as J.\ Geom.\ Anal.\ {\bf 36}, 141 (2026)}
\hfil
}

\maketitle

\section{Introduction and summary}

A Finsler manifold~\cite{br,pf}
is a natural generalization of a Riemannian manifold
in which the tangent space at each point is assigned a Minkowski norm.
The canonical Finsler inner product depends 
both on location in the manifold and direction in the tangent space.
The infinitesimal distance between two neighboring points
is determined by a homogeneous symmetric differential 2-form.
In particular,
we can measure the length of a path (independent of 
parametrization) in a Finsler manifold and study geodesics.
Riemannian geometry is the limiting case 
where the canonical inner product is Euclidean
and independent of direction,
with the infinitesimal distance determined 
by a symmetric quadratic differential 2-form.

A smoothly varying set of Minkowski norms over a manifold
is called a Finsler norm $F$.
A central challenge is to determine whether two given Finsler norms
correspond to distinct geometries.
Geometrical characterizations of certain Finsler norms are known. 
Cartan identified the Riemannian subset of Finsler manifolds 
via a symmetric torsion 3-tensor $\bfC$
constructed from $F$ and its derivatives~\cite{ec},
with $\bfC$ vanishing if and only if the Finsler manifold 
is a Riemannian manifold~\cite{ad}.
Similarly,
the Randers subset~\cite{gr} of Finsler manifolds
constructed from a Riemannian metric and a 1-form
is characterized by the symmetric Matsumoto 3-tensor $\bfM$~\cite{mm},
which vanishes iff the Finsler manifold is a Randers manifold~\cite{mh}.

The geodesics in a large class of Finsler manifolds 
have recently been identified with continuations 
of the classical trajectories of relativistic spinor wave packets 
interacting with a background~\cite{ak11}.
These spaces provide natural generalizations of  Randers geometry,
which is known to govern the Zermelo problem of navigation 
in the presence of a background wind~\cite{ez,zs,dbcr}.
In the physics context,
these Finsler manifolds originate 
from the extension of Einstein's general relativity 
to incorporate explicit violation of local Lorentz symmetry,
which can induce a clash between dynamics and geometry
and hence present an obstacle to a Riemannian description
of the geometry~\cite{ak04}.
Identifying characteristic tensors that can distinguish 
members of this class of Finsler manifolds is an intriguing open challenge,
and this work represents an initial effort to address it.

An interesting subset of these spaces are the bipartite spaces,
for which the Finsler norm $F$ is built as the sum or difference 
of a Riemannian norm $\rho$ and a seminorm $\sigma$ 
constructed from a symmetric nonnegative 2-form $\bfs$~\cite{ak11}.
The comparatively simple nature of this structure
makes it readily amenable to investigation,
while incorporating many cases of relevance. 
Special examples of bipartite spaces
include the $\bfa$ spaces~\cite{ak11},
which are related to Randers spaces;
the $\bfb$ and ${\bf ab}$ spaces~\cite{ak11},
which are related to $(\alpha,\beta)$ spaces~\cite{mm92,cs05};
certain $\bf H$ spaces,
some of which are isomorphic to ${\bf ab}$ spaces~\cite{krt12};
the ${\bf face}$ spaces~\cite{kr10},
which are also related to Randers spaces;
certain $\bf d$ and $\bf g$ spaces~\cite{rs18};
some $\bf k$ spaces~\cite{ek18};
and spaces with specific $\bfs$~\cite{sma19}.
A bipartite space has a nontrivial slit $S$
that includes the 0-section as usual for a Finsler manifold,
along with additional points that are fixed under scaling 
and hence lack the homogeneity of a Minkowski norm.
In the physics context,
the presence of these fixed points originates
in degeneracies of the quartic algebraic variety
governing the dynamics of the motion,
which in turn stem from the spinor nature of the particle.
The presence of the fixed points can affect the convexity of a bipartite space,
leading to points where the fundamental tensor 
has one or more nonpositive eigenvalues.

In the present work,
we introduce the notions of 
almost Finsler manifold and partial Finsler manifold,
which provide a mathematical basis 
for treating the relevant class of Finsler manifolds with nontrivial slits. 
As a simple application,
the formalism is used to clarify 
the relationships mentioned above
that link the $\bfa$ and $\bfb$ bipartite spaces 
to Randers spaces in $n$ dimensions.
The Matsumoto tensor $\bfM$ vanishes for the $\bfa$ spaces~\cite{ak11},
and we show here that the union of the $\bfa$-space indicatrices
matches the union of the Randers indicatrices.
The $\bfa$ spaces are thus examples of almost Finsler manifolds
that obey the same Matsumoto condition as Randers manifolds.
For $n=2$,
the union of the $\bfb$-space indicatrices also matches 
those of the $\bfa$ spaces and the Randers spaces,
but for $n>2$ the union of $\bfb$-space indicatrices 
is a distinct space. 

The main result of this work is the derivation 
of a symmetric 3-tensor $\bfS$ 
that vanishes for the almost Finsler manifolds 
that are bipartite spaces.
Since a Finsler geometry is specified by its Finsler norm $F$,
any geometric object must be expressible
in terms of $F$ and its derivatives. 
The indicatrix $F=1$ represents the level curve 
for the 1-homogeneous geometry,
so its shape is a geometric marker for the Finsler manifold.
This implies that repeated differentiation of the indicatrix
yields conditions that characterize the geometry,
but expressing these in geometric terms
is typically challenging because $F$ is nonpolynomial.
Here,
we instead identify suitable polynomials or near polynomials of $F$
having derivatives that can be identified with geometric quantities.
For bipartite spaces,
where the level curves of $F$ obey a quartic equation
rather than the simpler quadratic equation for Randers level curves,
one might expect up to four derivatives of $F$ to appear.
Remarkably,
this complication can be avoided for bipartite geometries
by taking advantage of the existence of the Riemann norm $\rho$.

Our procedure yields a symmetric 3-tensor $\bfS$
that vanishes for bipartite spaces
and that can be expressed using the geometric quantities
$F$ and $\Delta = F- \rho$, 
along with their derivatives.
Relevant objects include the Cartan torsion $\bfC$
and the mean Cartan torsion $\bfI$,
the angular metric $\bfh$ and the fundamental tensor $\bfg$,
and a quantity $\kappa$ that can be expressed
in terms of $F$, $\rho$ and other geometric objects.
Stated in the local coordinates defined in Sec.~\ref{Tensors}
and denoting a cyclic sum over three indices by $\sum_{(jkl)}$,
we prove the following theorem for bipartite spaces.

\begin{theorem}
The characteristic tensor 
$$
S_{jkl} = C_{jkl} 
- \frac{1}{\kappa}
\sum_{(jkl)}
\left(I_j + \frac{F^2}{(F-\Delta)\Delta} g^{kl}\Delta_k\Delta_{lj}\right)
\left(h_{kl} - \frac{F^2}{\Delta} \Delta_{kl}\right)
$$
vanishes for almost Finsler manifolds that are bipartite spaces.
\label{bipthm}
\end{theorem}

A particularly interesting subset of bipartite spaces 
consists of the $\bfb$ spaces,
which offer a natural complement to the $\bfa$ spaces and Randers spaces.
The Finsler norm for a $\bfb$ space 
is constructed from a Riemannian norm $\rho$ 
and a seminorm $\sigma$ 
that involves a perpendicular projection with $\bfb$,
in contrast to the parallel projection relevant 
for the $\bfa$ spaces and Randers spaces.
The simplicity of this structure permits 
explicit analytical derivations of geometrical objects 
such as the Cartan torsion, 
the Matsumoto tensor,
and the spray coefficients for geodesics,
as well as establishing results 
such as the demonstration that any $\bfb$ space
with Riemannian-parallel $\bfb$ is a Berwald space~\cite{ak11}.
The $\bfb$ spaces appear in certain physical scenarios,
such as a bead sliding on a wire
or a transversely magnetized chain~\cite{fl15}.
Their continuations are also central in effective field theories 
with violations of Lorentz invariance,
where they arise from dominant spin-dependent effects
breaking the discrete symmetry CPT~\cite{ck97}
and as such have been the subject of numerous experiments~\cite{tables}.

For $\bfb$ spaces,
the characteristic tensor $\bfB$
expressed in terms of geometric quantities
is comparatively elegant.
In the final section of this work,
using the local coordinates defined in Sec.~\ref{Tensors},
we prove the following theorem for $\bfb$ spaces.

\begin{theorem}
The characteristic tensor 
$$
B_{jkl} = C_{jkl} 
- \frac{1}{\kappa_b}
\sum_{(jkl)}
I_j \left(h_{kl} - \frac{F^2}{\Delta} \Delta_{kl}\right)
$$
vanishes for almost Finsler manifolds that are $\bfb$ spaces.
\label{bthm}
\end{theorem}

Theorems \ref{bipthm} and \ref{bthm}
offer support for the conjecture 
that the geometry of almost Finsler manifolds 
can be classified by their geometric properties
as encoded in suitable symmetric 3-tensors.
For example, 
the vanishing of the Cartan torsion $\bfC$ guarantees 
that a Finsler manifold is Riemannian~\cite{ad}
and the vanishing of the Matsumoto tensor $\bfM$
guarantees that a Finsler manifold is Randers~\cite{mh}.
Note that the tensors $\bfC$, $\bfM$, $\bfS$, and $\bfB$ 
are vertical invariants that depend only on the tangent-space geometry.
In contrast,
3-tensors involving covariant derivatives 
such as the Landsberg tensor
measure the behavior of the tangent-space geometry 
while moving from point to point.  
Further evidence for the classification conjecture could be adduced
from other calculable examples
associated with the motion of a fermion wavepacket in a background,
such as the special bipartite spaces mentioned
above~\cite{krt12,kr10,rs18,ek18,sma19}.
A general proof of the classification conjecture
would evidently be of definite interest
and is likely to require a deeper understanding
of the relationships between geometric elements
and derivatives of level curves of almost Finsler manifolds. 

Note that Theorems \ref{bipthm} and \ref{bthm}
are one-way results:
the relevant characteristic tensors are proved to vanish
for the designated almost Finsler manifolds.
Based on the interpretation 
of the vanishing of the characteristic tensors $\bfS$ and $\bfB$
as differential equations that determine almost Finsler norms, 
it is plausible to conjecture that the converses 
of Theorems \ref{bipthm} and \ref{bthm} are also valid,
perhaps modulo a small subset of degenerate or singular cases 
arising in special limits.
If true,
this conjecture would extend Theorems \ref{bipthm} and \ref{bthm}
to biconditional statements.
A proof would be of definite interest 
but lies outside our present scope.

The spaces considered here are continuations 
of Lorentz-Finsler spaces,
which have Minkowski pseudonorms with Lorentz signature,
to almost Finsler manifolds and partial Finsler manifolds
having Minkowski norms with the standard Euclidean signature.
Identifying a satisfactory definition of Lorentz-Finsler manifolds 
remains a subject of active 
research~\cite{jb70,gsa85,ma94,bf00,ak11,pw11,lph12,cm15,cs16,em17,js20,hpv22}
with prospective applications in gravity, astrophysics,
and condensed matter~\cite{lp18,bv18,kl21,mt21,%
rl21,jsv22,klmss22,zm23,bs23,hyzl24,ncnpm24,st25,mcw12,ms25,pvfp25}.
Our constructions of the characteristic tensors $\bfS$ and $\bfB$
rely on derivatives of the almost Finsler norms and partial Finsler norms
and hence can generically be extended from norms to pseudonorms.
It is therefore reasonable to conjecture
that our treatment could be adapted to form the basis 
for a definition of almost Lorentz-Finsler manifolds
that would permit the extension of Theorems \ref{bipthm} and \ref{bthm}
to the corresponding bipartite and $\bfb$ Lorentz-Finsler spaces.
One indication in support of this conjecture
is that the definition of a pseudo-Finsler manifold
as considered in Ref.~\cite{hpv22}
requires that the fundamental tensor $\bfg$ is nondegenerate
and hence is intermediate between a partial Finsler manifold
and an almost Finsler manifold.
In any event,
establishing the validity of the conjecture has the potential
for impact in both the mathematical and the physical contexts. 

The remainder of this work is structured as follows.
In Sec.~\ref{Almost and partial Finsler manifolds},
we introduce the notions of almost Finsler manifold
and partial Finsler manifold,
and we discuss some aspects of the associated tensor calculus.
Section~\ref{Bipartite spaces}
considers bipartite spaces 
and demonstrates relationships among the bipartite indicatrix unions
of the Randers spaces, the $\bfa$ spaces, and the $\bfb$ spaces. 
The proofs of Theorems \ref{bipthm} and \ref{bthm}
are provided in Sec.~\ref{Characteristic tensors}.

\section{Almost and partial Finsler manifolds}
\label{Almost and partial Finsler manifolds}

In this section,
we introduce the definitions 
of an almost and a partial Minkowski norm on a vector space
and an almost and partial Finsler norm on a smooth manifold, 
giving examples along the way.   

\subsection{Minkowski norms}
By vector space we always mean a finite-dimensional real vector space.
A subset $S$ of $V$ is {\em conelike} 
if $S = \R S$, i.e., $S$ is closed under scalar multiplication.

\begin{definition}
\label{pms}
A {\em partial Minkowski space} is a triple $(V,S,F)$,
where $V$ is a vector space, 
$S \subset V$ is a closed, conelike, nonempty subset, 
and $F : V \backslash S \to (0,\infty)$ is a smooth function satisfying 
{\em homogeneity}: 
$F(\lambda \bfy) = \lambda F(\bfy)$ 
for $\lambda > 0$ and $\bfy \in V \backslash S$.  The function
$F$ is called a partial Minkowski norm. 
\end{definition}

\begin{definition}
\label{ams}
An {\em almost Minkowski space} is a partial Minkowski space $(V,S,F)$ 
that also satisfies 
{\em positive definiteness}:
for all $\bfy \in V \backslash S$, 
the symmetric bilinear form 
\begin{align*}
\bfg(\bfy)  : V \times V & \to \R \ , \\
(\bfu, \bfv) & \mapsto \tfrac{1}{2} 
\lim_{s,t \to 0} \frac{F^2(\bfy + s\bfu + t\bfv)}{st}
\end{align*}
is positive definite.
$F$ is then called an almost Minkowski norm. 

A {\em Minkowski space}  is an almost Minkowski space $(V,S,F)$ 
with $S = \{0\}$.
\end{definition}

Note that $\bfg$ is radially constant:
$\bfg(\lambda \bfy) = \bfg(\bfy)$ 
for $\lambda > 0$ and $\bfy \in V- S$.

\begin{definition}
Let $(V,S,F)$ be a partial Minkowski space.
Its {\em indicatrix} is the level hypersurface
$$
F^{-1}\{1\} = \{\bfv \in V \backslash S \mid F(\bfv) = 1\}.
$$
The {\em solid indicatrix} is
$F^{-1}[0,1]$.
\end{definition}
The homogeneity condition implies that the indicatrix determines $F$.
The solid indicatrix is a smooth manifold whose boundary is the indicatrix.
If $F$ is a Minkowski norm, 
the solid indicatrix is strictly convex, 
is diffeomorphic to a closed disk, 
and contains the origin.

Let $(V,\langle , \rangle)$ be a finite-dimensional real inner-product space.
The {\em Euclidean norm} 
$\|\bfy\| = \sqrt{\langle \bfy,\bfy \rangle}$ 
is a Minkowski norm.
The Euclidean norm determines the inner product
by the polarization identity 
$\langle \bfv, \bfw \rangle = 
(\| \bfv + \bfw \|^2 - \|\bfv\|^2 - \| \bfw\|^2)/2$.

\begin{example}
Let $\bfa \in V$ be a vector whose length is less than one.
Then $F{_\bfa}(\bfy) = \|\bfy \| + \langle \bfa, \bfy \rangle$ 
is a Minkowski norm, 
called a {\em Randers norm}.  
Alternatively, 
if $\alpha : V \to \R$ is a functional with norm less than one, 
then $F_\alpha(\bfy) = \|\bfy \| + \alpha(\bfy)$ is a Minkowski norm.
The {\em Randers spaces} associated with $\bfa$ are the Minkowski spaces 
$(V, F_{\bfa})$ and $(V, F_{-\bfa})$.
\end{example}

\begin{example}
Recall that for a given nonzero vector $\bfv \in V$ 
and for any $\bfy \in V$ there is a unique expression
\begin{equation}
\bfy = \bfy_{\|} + \bfy_{\perp} \ ,
\label{vecdecomp}
\end{equation}
with $\bfy_{\|}$ parallel to $\bfv$ 
and $\bfy_{\perp}$ perpendicular to $\bfv$.

Let $\bfa \in V$ be a given nonzero vector whose length is less than one.
The {\em $\bfa$ spaces} are the partial Minkowski spaces 
$(V, (\R \bfa)^\perp, F_{\bfa}^+)$ and $(V, (\R \bfa)^{\perp}, F_{\bfa}^-)$,
where $F_{\bfa}^{\pm}(\bfy) = \|\bfy\| \pm \|\bfa\|  \|\bfy_{\|}\|$.
These $\bfa$ norms are in correspondence with the Randers norm
in the following sense:
if $\langle  \bfa, \bfy \rangle > 0$, 
then $F_{\bfa}^{\pm}(\bfy) = F_{\pm \bfa}(\bfy)$,
and if $\langle \bfa, \bfy \rangle < 0$, 
then $F_{\bfa}^{\pm}(\bfy) = F_{\mp \bfa}(\bfy)$.
The $\bfa$ spaces are therefore almost Minkowski spaces.
Note that,
although the functions $F^\pm_{\bfa} : V \to \R$ are continuous, 
they are not smooth on the hyperplane $(\R\bfa)^{\perp}$.  
One can see this by setting 
$\bfa = c ~\bfe_n$ where $-1 < c < 1$, $c \not = 0$, 
and then computing the derivative. 
\label{vecdecompa}
\end{example}

\begin{example}
Let $\bfb \in V$ be a given nonzero vector whose length is less than one.
The {\em $\bfb$ spaces} are the partial  Minkowski spaces 
$(V, \R \bfb, F_{\bfb}^+)$ and $(V, \R \bfb, F_{\bfb}^-)$,
where $F_{\bfb}^{\pm}(\bfy) = \|\bfy\| \pm \|\bfb\|  \|\bfy_{\perp}\|$.
The functions $F^\pm_{\bfb} : V \to \R$ are continuous
but are not smooth on the line $\R\bfb$. 
It is known that $F_{\bfb}^+$ is always 
an almost Minkowski norm~\cite{ak11}. 
Corollary \ref{fminus} below
reveals that $F_{\bfb}^-$ is an almost Minkowski norm only for $\dim V = 2$.
\label{vecdecompb}
\end{example}

A Minkowski space is {\em reversible}
if the Minkowski norm $F(\bfy)$ satisfies
$F(-\bfy)=F(\bfy)$.
The Randers spaces are nonreversible,
while the $\bfa$ and $\bfb$ spaces are reversible.

For a Minkowski norm $(V,F)$, 
the map $F^2$ extends to a smooth map $V \to \R$ 
if and only if the Minkowski norm is Euclidean. 
(If $F^2$ extends to a smooth map, 
the second derivative $\bfg(\bfy)$ must be constant 
since it is radially constant and extends continuously to the origin).
For $\dim V > 1$, 
the almost Minkowski $\bfa$ and $\bfb$ spaces 
do not extend to Minkowski spaces.
This fact and the presence of the $\bfa$, $\bfb$, 
and other spaces with similar features in the physics literature
are the motivation for our introduction 
of the notions of an almost Minkowski space and a partial Minkowski space. 

The Randers spaces, $\bfa$ spaces, and $\bfb$ spaces 
have complementary aspects
originating from the roles of $\bfy$, $\bfy_\|$, and $\bfy_\perp$
in the corresponding structures $F$~\cite{ak11}.
Some aspects of this complementarity and the special case $\dim V = 2$ 
are discussed in Sec.~\ref{Bipartite indicatrix unions}.

\begin{definition}
Let $\bfa \in \R^n$ and $\bfr \in \R^n_{> 0}$.
Then 
$$
E = \{\bfy \in \R^n \mid \frac{(y^1-a^1)^2}{(r^1)^2} + \cdots 
+ \frac{(y^n-a^n)^2}{(r^n)^2} = 1\}
$$
is an {\em ellipsoid in $\R^n$ with center $\bfa$}.  

Let $E$ be a subset of a vector space $V$.
$E$ is an {\em ellipsoid in $V$ with center $\bfa \in V$} 
if there exists an isomorphism of vector spaces $f : V \to \R^n$ 
so that $f(E)$ is an ellipsoid in $\R^n$ with center $f(\bfa)$.  
\end{definition}

If $E\subset V$ is an ellipsoid with center $\bfa$, 
then for any isomorphism of vector spaces 
$f : V \to \R^n$, $f(E)$ is an ellipsoid in $\R^n$ with center $f(\bfa)$.  

Note that we now have two definitions of the concept 
of ellipsoid with center $\bfa$ for $V = \R^n$: 
one using coordinates and one using the vector space structure.
The two definitions coincide.   

The indicatrices $I^{\pm}_{\bfa}$ for the Randers spaces $(V, F_{\pm\bfa})$ 
are ellipsoids with center  $\mp\bfa/(1- \|\bfa\|^2)$.
Note that the center is not the origin if $\bfa$ is nonzero.

\subsection{Finsler manifolds}

A subset $S$ of the total space $E$ of a vector bundle $V$ 
is {\em conelike} if $S = \R S$.

\begin{definition}
A {\em partial (respectively, almost) Finsler manifold} $(M,S,F)$ 
is a triple where $M$ is a smooth manifold, 
$S \subset TM$ is a closed, conelike subset 
which contains the zero section of $M$, 
and $F : TM \backslash S \to (0,\infty)$ is a smooth function
such that for all $x \in M$, $(T_xM, S \cap T_xM, 
F|_{x})$ is a partial (respectively, almost) Minkowski norm.   
Here, 
$F|_{x}$ indicates $F$ restricted to $T_xM \backslash S\cap T_xM$.
A {\em Finsler manifold} is a pair $(M, F : TM \backslash M \to (0,\infty))$
satisfying homogeneity and positive definiteness in each tangent space.
Here, the subset of zero vectors in $TM$ is identified with $M$.
\label{afmslit}
\end{definition}

The purpose of homogeneity is to establish a well-defined length of a curve, 
independent of parametrization.
The purpose of positive definiteness is to allow 
the application of concepts from Riemannian geometry
as in Ref.~\cite{bcs}.

The manifold $TM \backslash S$ is called 
the {\em slit tangent bundle} 
and $S$  is called the {\em slit}.  

If $(M,F)$ is a Finsler manifold so that
 $F|_{x}$ is a Euclidean norm for all $x \in M$, 
the pair $(M,F)$ is a Riemannian manifold. 
In this case we use the letter $\rho$ instead of $F$
and denote the corresponding inner product 
by $\langle \bfy, \bfy' \rangle_\rho$. 

If $(M,S,F)$ is a partial Finsler manifold,
define the {\em extended slit} 
$$ES = S \cup \{\bfy \in TM \backslash S \mid \bfg(\bfy) 
\text{ is not positive definite}\}.$$
Then $(M,ES,F)$ is an almost Finsler manifold, 
called the {\em truncation} of the partial Finsler manifold $(M,S,F)$.

\begin{example}
Let $(M,\rho)$ be a Riemannian manifold.
Let $\pi: TM \to M$ be the projection map.

Let $\bfa : M \to TM$ be a vector field on $M$ 
such that $\rho(\bfa(x)) < 1$ for all $x \in M$.
Then $(M,F_{\pm\bfa})$ 
with
\begin{equation}
F_{\pm \bfa}(\bfy) = 
\rho(\bfy) \pm \langle \bfa(\pi(\bfy)),\bfy \rangle_\rho
\label{rmfd}
\end{equation}
are Finsler manifolds, 
called {\em Randers manifolds}.  

The partial Finsler manifolds 
$(M, \cup_{x \in M} (\R \bfa(x))^{\perp}, F_{\bfa}^{\pm})$
with 
\begin{equation}
F_{\bfa}^{\pm}(\bfy) = 
\rho(\bfy) \pm \rho(\bfa(\pi(\bfy)) \rho(\bfy_{\|}) 
\label{amfd}
\end{equation}
are called {\em $\bfa$ manifolds}.
Note that $F_{-\bfa}^{\pm} = F_{\bfa}^{\pm}$ 
and $F_{\bfa}^{\pm}(\bfy) = F_{\bfa}^{\pm}(-\bfy)$.
If $\langle \bfa(\pi(\bfy)), \bfy \rangle_\rho > 0 $, 
then $F_{\bfa}(\bfy) = F^+_{\bfa}(\bfy)$ 
while if $\langle \bfa(\pi(\bfy), \bfy \rangle_\rho <0$,
then $F_{\bfa}(\bfy) = F^-_{\bfa}(\bfy)$.

Let $\bfb: M \to TM$ be a vector field on $M$ 
so that $\rho(\bfb(x)) < 1$ for all $x \in M$.
The partial Finsler manifolds 
$(M, \cup_{x \in M} \R \bfb(x), F_{\bfb}^{\pm})$
with 
\begin{equation}
F_{\bfb}^{\pm}(\bfy) = 
\rho(\bfy) \pm \rho(\bfb(\pi(\bfy)) \rho(\bfy_{\perp}) 
\label{bmfd}
\end{equation}
are called {\em $\bfb$ manifolds}.
Note that $F_{-\bfb}^{\pm} = F_{\bfb}^{\pm}$ 
and $F_{\bfb}^{\pm}(\bfy) = F_{\bfb}^{\pm}(-\bfy)$.
\label{abmfd}
\end{example}

The results of \cite{krt12} 
combined with Lemma~\ref{aminusmemma} and Corollary~\ref{fminus} below
reveal that the $\bfa$ manifolds and $\bfb$ manifolds
are almost Finsler manifolds,
except for $(M, \cup_{x \in M} \R \bfb(x), F_{\bfb}^{-})$ 
with $\dim M >2$.  An interesting open problem is 
to determine the extended slit for these manifolds.

\section{Tensors}
\label{Tensors}

We review the notion of tensor fields on a smooth manifold, 
in preparation for the computations in the next section.
We give a novel interpretation of the higher derivatives
of a real-valued function on the tangent bundle of a smooth manifold.

\subsection{Tensors and tensor fields}

A {\em $(p,q)$-tensor $\bfT$ on a vector space $V$} 
is an element of $T^{p,q}V = V^{\otimes p} \otimes V^{*\otimes q}$. 
This is the same as a multilinear function 
$
 (V^*)^p \times V^q = 
V^* \times \cdots \times V^* \times V \times \cdots \times V \to \R.
$
A {\em metric} $\bfg$ on $V$ is given 
by a $(0,2)$-tensor field that is symmetric 
(invariant under the interchange map $V^* \otimes V^* \to V^* \otimes V^*$) 
and positive definite (for $\bfv \not = 0$, $\bfg(\bfv, \bfv) > 0$).   
A {\em $(p,q)$-tensor field on a manifold $M$} 
is a smooth section of $T^{p,q}M = TM^{\otimes p} \otimes TM^{*\otimes q}$.
If $U$ is an open subset of $V$
and $f: U \to \R$ is a smooth function, 
then the derivatives of $f$,
denoted by $\bff', \bff^{''},\bf f^{'''}, \dots$, 
are symmetric tensor fields on $U$ of type $(0,1), (0,2), (0,3), \dots$.  

\begin{example} \label{vs_metric_Cartan}
If $F$ is a Minkowski norm on a vector space $V$, 
then the second derivative of $F^2/2$ 
is given by $\bfg(\bfy)$ in Definition \ref{ams}.
The third derivative of $F^2/2$ is called the {\em Cartan tensor} 
and is given by the symmetric trilinear function 
$\bfC(\bfy) : V \times V \times V \to \R$.  
The Cartan tensor vanishes if and only if $F$ is a Euclidean norm~\cite{ad}.

A Riemannian metric $\bfg$ gives a $(0,2)$-tensor field 
on the underlying manifold.
\label{defs}
\end{example}

\begin{definition}
Let $E \to M$ be a smooth bundle over a smooth manifold.
A {\em tensor field of type $(p,q)$ on $M$ with values in $E$} 
is a smooth section $\bfT$ of the bundle  
$$ 
E \otimes \cdots \otimes E \otimes E^* \otimes \cdots \otimes E^* \to M. 
$$
In effect, 
$\bfT(x)$ is a $(p,q)$-tensor on $E_x$ 
that depends smoothly on $x \in M$.
The vector space of all $(p,q)$-tensors with values in $E$ 
is denoted by $T^{p,q}(E \to M)$,
which we abbreviate as $T^{p,q}E$.
A {\em metric on $E$} is a tensor field of type $(0,2)$ 
that gives a metric at each fiber $E_x$.
\end{definition}

Let $\pi: TM \to M$ be the tangent bundle of a manifold $M$.
Let $W \subset TM$ be an open subset. 
Consider the pullback diagram
\begin{equation} 
 \begin{CD}
(\pi|_W)^*TM @>>> TM \\
@V\pi_1VV @VV\pi V \\
W @>\pi|_W >> M
\end{CD}
\label{pb_abstract}
\end{equation}
where
$$(\pi|_W)^*TM =\{ (\bfy, \bfy') 
\in W \times TM \mid \pi(\bfy) = \pi(\bfy')\}$$ 
is the pullback of $\pi$ under $\pi|_W$.
Then $\pi_1$ 
(the projection on the first factor) 
is a rank-$\dim M$ vector bundle with fibers $ T_{\pi(\bfy)}M$.
Furthermore,
$\pi_1$ has a canonical section $l(\bfy) = (\bfy, \bfy)$.
The derivatives of a smooth function $f : W \to \R$, 
denoted by $\bff', \bff^{''}, \bff^{'''}, \dots$,
are symmetric tensor fields on $W$
with values in $(\pi|_W)^*TM$ of type $(0,1), (0,2), (0,3), \dots$
Here, 
for instance, 
if $\bfy \in W$ and $x = \pi(\bfy)$ then
\begin{align*}
\bff''(\bfy)  &: T_xM \times T_xM  \to \R \ , \\
\bff''(\bfy)(\bfu, \bfv) 
&=  \lim_{s,t \to 0} \frac{f(\bfy + s\bfu + t\bfv)}{st} \ .
\end{align*}

\begin{example} \label{metric_and_Cartan}
Let $(M,S,F)$ be a partial Finsler manifold.  
Consider the pullback diagram
\begin{equation}
\begin{CD}
\pi|^*TM @>>> TM \\
@V\pi_1VV @VV\pi V \\
TM \backslash S @>\pi|>> M
\end{CD}
\label{pb}
\end{equation}
The {\em fundamental tensor} $\bfg$
is the second derivative of $F^2/2$. 
It is a symmetric positive-definite tensor field 
of type $(0,2)$ on $TM \backslash S$ with values in $\pi|^*TM$.
Thus $\pi_1$ is a bundle with metric $\bfg$.
The {\em Cartan tensor} $\bfC$ is the third derivative of $F^2/2$ 
and is a totally symmetric trilinear $(0,3)$-tensor.
Note that if $\pi(\bfy) = x$ 
then $\bfg(\bfy)$ and $\bfC(\bfy)$ 
are the tensors on the vector space $T_xM$ 
in Example \ref{vs_metric_Cartan}.
Note also that 
$\bfC(\bfy) : T_x M \times T_xM \times T_xM \to \R$ 
is a totally symmetric trilinear function
when $\pi(\bfy) = x$. 

The Cartan tensor vanishes if and only if $(M,F)$ 
is a Riemannian manifold.
\end{example}

\subsubsection{Components and coordinates}

We first consider components of a tensor, 
then coordinates for manifolds and vector bundles, 
and finally use coordinates to express the components of a derivative tensor.

Let $\{\bfe_1, \cdots, \bfe_n\}$ be an ordered basis for a vector space $V$ 
and let  $\{\bfe^1, \cdots, \bfe^n\}$ be the dual basis of  $V^*$, 
so that $\bfe^j(\bfe_k) = 1, j = k$ and $\bfe^j(\bfe_k) = 0, j\neq k$. 
One can represent an element of $V$ 
by $\sum a^j\bfe_j$ or $a^j\bfe_j$ or $a^j$ 
and an element of $V^*$ 
by $\sum a_j\bfe^j$ or $a_j\bfe^j$ or $a_j$.
A $(p,q)$-tensor $\bfT$ can be written as 
\begin{equation} 
\bfT = \sum T^{j_i \cdots j_p}{}_{k_1 \cdots k_q} 
\bfe_{j_1} \otimes \cdots \otimes \bfe_{j_p} 
\otimes 
\bfe^{k_1} \otimes \cdots \otimes  \bfe^{k_q} \ ,
\label{tensor_coordinates}
\end{equation}
where $1 \leq j_1, \dots, j_p\leq n$ 
and $1 \leq k_1,  \dots,  k_q\leq n$.
The quantities 
$T^{j_i \cdots j_p}{}_{k_1 \cdots k_q}$
are called the {\em components of the tensor} $\bfT$.
 
Conversely, 
a collection of real numbers 
$T^{j_i \cdots j_p}{}_{k_1 \cdots k_q}$
defines a tensor $\bfT$ of type $(p,q)$.
 
\begin{example}
There is an isomorphism $T^{1,1}V \to \Hom(V,V)$ 
given by $\bfv \otimes \omega \mapsto (\bfv' \mapsto \omega(\bfv')\bfv)$.
Let $\iota$ be the $(1,1)$-tensor that maps to $\id_V$.
The components of $\iota$ with respect to {\em any} basis are
$\bfe^j(\bfe_k) = \delta^j{}_k$, 
where the collection of real numbers $\delta^j{}_k$ 
is called the {\em Kronecker delta}. 
Note that $\delta^j{}_k=1$ when $j=k$ and $\delta^j{}_k = 0$ when $j\neq k$. 
\end{example}

Let $E \to M$ be a smooth vector bundle over a smooth manifold.
Suppose there are sections $\{\bfe_1, \cdots, \bfe_n\}$ of the bundle 
so that $\{\bfe_1(x), \dots, \bfe_n(x)\}$ 
is a basis for the fiber $E_x$ for every $x \in M$.
(Every bundle is locally trivial, so this can be arranged
by passing to a neighborhood of a point.)  
A tensor field $\bfT$ of type $(p,q)$ with values in $E$ 
can be written as in Eq.~\eqref{tensor_coordinates},
where the tensor components 
$$
T^{j_i \cdots j_p}{}_{k_1 \cdots k_q} : M \to \R
$$
are smooth functions.

For an  open set $U$ in $\R^n$ and a smooth function $f : U \to \R$, 
let $f_j : U \to \R$, $j = 1, \dots, n$ denote the partial derivatives.

Let $M$ be a smooth $n$-manifold.
For every $x \in M$, 
there is an open neighborhood $U$ of $x$ 
and {\em coordinate functions} $x^j: U \to \R$ for $ j = 1, \dots, n$ 
so that $(x^1, \dots, x^n) : U \to \R^n$ is a smooth embedding.
One then has tensor fields $\partial_j$ and $dx^j$ on $U$ 
of type $(1,0)$ and $(0,1)$,
respectively, 
that push forward and pull back from the coordinate tensor fields on $\R^n$.    
If $f : U \to \R$ is a smooth function, 
define the partial derivatives $f_j : U \to \R$ 
by $(f\circ (x^1, \dots, x^n)^{-1})_j \circ (x^1, \dots, x^n)$.
These partial derivatives depend on the choice of coordinates; 
we also write $\partial f/\partial x^j$ or $\partial_jf$ for $f_j$.

Let $\pi : E \to B$ be a vector bundle of rank $n$.
For every $y \in B$, 
there is an open neighborhood $U$ of $y$ 
and linearly independent sections $\bfe_j : U \to E$ for $ j = 1, \dots, n$.
Define the {\em coordinate functions} 
$y^j : \pi^{-1}U \to \R$ by $\bfy = \sum y^j(\bfy) \bfe_j (\pi(\bfy))$.
The coordinate functions give a homeomorphism 
$(\pi, y^1, \dots, y^n): \pi^{-1}U \to U \times \R^n$.

Let $\pi: E \to M$ be a smooth vector bundle of rank $m$ 
over an $n$-manifold $M$.
Then for every $x \in M$, 
there exists an open neighborhood $U$ of $x$ 
with coordinate functions $\bar x_j$ 
and smooth linearly independent sections $\bfe_j$ defined on $U$.
The coordinate functions give a smooth embedding 
$(x^1, \dots, x^n, y^1, \dots y^m) : \pi^{-1}U \to \R^n \times \R^m$, 
where $x^j = \bar x^j \circ \pi$.  

Let $W \subset TM$ be open.   Recall the pullback diagram 
\begin{equation*} 
 \begin{CD}
\pi|^*TM @>>> TM \\
@V\pi_1VV @VV\pi V \\
W @>\pi| >> M
\end{CD}
\end{equation*}

If $f : W \to \R$ is a smooth function, 
we wish to have local formulas for the components of the derivatives 
$\bff', \bff'', \bff''', \dots$, 
which are  symmetric tensor fields on $W$ 
with values in $\pi|^*TM$ of type $(0,1), (0,2), (0,3), \dots$.  
For every $x \in \pi(W)$ there is a neighborhood $W$ and
linearly independent sections  $\bfe_j : W \to TM$ for $j =1 , \dots, n$.
Then $\pi_1$ has pullback sections $\pi|^*\bfe_j$, 
defined by $\bfy \mapsto (\bfy,\bfe_j(\pi(\bfy))$,
and
\begin{align*}
\bff' &= \sum  \frac{\partial f}{\partial y^j} \pi|^*\bfe^j \ , \\
\bff'' &= \sum  \frac{\partial^2 f}{\partial y^j\partial y^k}\pi|^*\bfe^j 
\otimes \pi|^*\bfe^k \ . \\
\end{align*}
Note that the partial derivatives of $f$ with respect to the $x^j$ 
are absent from the formulas for the derivatives, 
and we often abbreviate
$f_j =  \partial f/\partial y^j$ 
and $f_{jk} = \partial^2 f/\partial y^j\partial y^k$.

\subsubsection{Contraction; lowering and raising indices}

Let $j$ be an integer between 1 and $p$ 
and $k$ be a integer between 1 and $q$.
Then there is a linear map $c_{j,k} : T^{p,q}V \to T^{p-1,q-1}V$ 
given by evaluating the $(p+k)$-factor on the $j$-th factor.
This map is called  {\em $(j,k)$-contraction}.   

There is a tensor product isomorphism 
$$
- \otimes - : T^{p,q}V \otimes T^{p',q'}V \to T^{p+p',q+q}V .
$$

For the rest of this section, 
assume we have a metric $\bfg \in T^{0,2}V$ on $V$.
Then there is a duality isomorphism~$\vee: V \to V^*$ 
defined by $\bfv^\vee(\bfv') = g(\bfv,\bfv')$.
We also write $\vee$ for the inverse isomorphism.
Then we have isomorphisms
$$
\vee : T^{p,q}V \to T^{q,p}V , 
$$
so that 
$(\bfv_1 \otimes \cdots \otimes \bfv_p 
\otimes \omega^1 \otimes \cdots \otimes \omega^q)^\vee = 
(\omega^1)^\vee \otimes \cdots \otimes (\omega^q)^\vee 
\otimes \bfv_1^\vee \otimes \cdots \otimes \bfv_p^\vee$.  
In the physics literature,
the dual of a metric is called the inverse metric
$\bfg^{\vee} \in T^{2,0}V$, 
for reasons that become clear below.

For an integer $j$ between 1 and $p$ 
and for an integer $k$ between 1 and $q$,
we can define isomorphisms
\begin{align}
l_j = c(j,2) \circ (g^\vee \otimes -)&: T^{p,q}V \to T^{p-1,q+1}V ,
\nonumber\\
r_k  = c(2,k) \circ (g \otimes -)&: T^{p,q}V \to T^{p+1,q-1}V ,
\label{isolr}
\end{align}
called {\em lowering the $j$-index} 
and {\em raising the $k$-index},
respectively.
One can verify that $r_1(g) = r_2(g) = \iota = l_1(g^\vee) = l_2(g^\vee)$.

\begin{example} 
Now suppose our vector space has 
an ordered basis $\{\bfe_1, \dots, \bfe_n\}$.
Then we can write all the above in components.
The components of a contraction of a $(p+1,q+1)$-tensor are given by 
$$
c_{a,b}(\bfT)^{j_1 \cdots j_p}{}_{k_1 \cdots k_q} = 
\sum_{l=1}^{n} T^{j_1 \cdots j_{b-1} l j_{b} \cdots j_p}
{}_{k_1 \cdots k_{a-1} l k_{a} \cdots k_q}.
$$
In what follows,
we omit the summation sign in the notation
(Einstein summation convention).
There is a special notation 
(popular in the physics community) 
for the components of a tensor after raising and lowering indices.
We illustrate with some examples.
First, 
suppose $\bfT$ is a $(2,0)$-tensor with components $T^{jk}$.
Then,
\begin{align*}
l_1\bfT & = T_j{}^k~ \bfe_k \otimes \bfe^j , 
\\
l_2\bfT & = T^j{}_k~ \bfe_j \otimes \bfe^k .
\end{align*}
Notice that $T^{lj}g_{lk} = T_k{}^j$ and $T^{jl}g_{lk} = T^j{}_k$.
Now suppose that $\bfT$ is a $(0,3)$-tensor.
Then
$$
r_3r_1\bfT = T^j{}_k{}^l \bfe_j \otimes \bfe_l \otimes \bfe^k ,
$$
with $T^j{}_k{}^l = T_{mkn}g^{jm}g^{ln}$.
By lowering an index on the dual metric,
one sees that $g^{jl} g_{lk} = \delta^j{}_k$, 
so the matrices $g^{jk}$ and $g_{kj}$ are inverses.
\end{example} 

One can apply all of the above to tensor fields by working fiberwise.
Let $E \to M$ be a smooth vector bundle.
Let $\bfT$ be a $(p,q)$-tensor field with values in $E$.
Then one can define contractions $c_{j,k}\bfT$.
Next suppose that $M$ has a Riemannian metric $\bfg$.
One can define  $l_j\bfT$ and  $r_j\bfT$, 
lowering and raising indices.
Finally, 
assume  that there are sections $\{\bfe_1, \cdots, \bfe_n\}$ 
of the bundle so that for every $x \in M$, 
$\{\bfe_1(x), \dots, \bfe_n(x)\}$ is a basis for the fiber $E_x$.
Then one can define components using the special notation above.

Our philosophy for the rest of the paper is 
{\em define globally, compute locally}!

\section{Bipartite spaces}
\label{Bipartite spaces}

We define bipartite spaces,
discuss some of their properties,
and give examples.

\subsection{Basics}

Let $(M,\rho)$ be a Riemannian manifold 
with $\rho(\bfy) = \sqrt{\bfr(\bfy,\bfy)}$; 
$\rho$ is the norm and $\bfr$ is the associated inner product.
An alternative notation for $\rho(\bfy)$ is $\|\bfy\|$
and for $\bfr(\bfy,\bfy)$ is $\langle \bfy,\bfy\rangle_\rho$.
Let $\bfs$ be a symmetric nonnegative $(0,2)$-tensor on $M$.
We assume that the eigenvalues of $\bfs$ are less than one 
and greater than or equal to zero.
More precisely, 
for every $x \in M$, let $s_x : T_xM \to T_xM$ 
be the unique endomorphism satisfying 
$\bfs(\bfy,\bfy') = \bfr(s_x(\bfy),\bfy')$ 
for all $\bfy,\bfy' \in T_xM$; 
then all eigenvalues of $s_x$ are assumed to be less than one 
and greater than or equal to zero.
Note that $\bfr(s_x(\bfy),\bfy') = \bfr(\bfy,s_x(\bfy'))$.

We call $\bfs$ the {\em bipartite form} and 
$\sigma(\bfy) = \sqrt{\bfs(\bfy,\bfy)}$ the {\em bipartite factor}.

The eigenvectors of $s_x$ with zero eigenvalue
span a subspace $S_x$ of $T_xM$.  Let $S = \cup_{x\in M} S_x \subset TM$.

\begin{theorem}
Let $F^{\pm} = \rho \pm \sigma$.
Then $(M,S,F^+)$ and $(M,S,F^-)$ are partial Finsler manifolds.  
\end{theorem}

\begin{proof}
To prove that $F^{\pm}(\bfy) > 0$, 
we use that $s_x$ is Hermitian with respect to the inner product $\bfr$ 
and the spectral-norm inequality 
$$
\| s_x(\bfy)\| \leq \|s_x\| \|\bfy\| \ ,
$$
where $ \|s_x\|$ is the absolute value of the maximal eigenvalue of $s_x$.
\end{proof}

The partial Finsler manifolds $(M,S,F^{\pm})$ are called {\em bipartite spaces}.
Below we examine the question of whether the fundamental tensor $\bfg$
is positive definite on these spaces
and hence whether they are almost Finsler manifolds.

Matching the terminology accompanying Definition \ref{afmslit},
$S$ is called the {\em slit}.
Let $I^{\pm} = (F^{\pm})^{-1}(1)$ be the indicatrices.  
The indicatrix $I^+$ is called the {\em lemon},
while the indicatrix $I^-$ is called the {\em apple} 
(see \hyperref[fig1]{Fig.~1} below and use your imagination).  

Let $S^f$ be the space of fixed points
of the indicatrix $I^{\pm} = (F^{\pm})^{-1}(1)$ 
under the scaling $\bfs \mapsto \lambda \bfs$, 
$\lambda \in (0,1)$,
and let $S^f_x = S^f \cap T_xM$.
Note that $S = \R S^f$.

\begin{lemma}
$S^f_x$ is diffeomorphic to a sphere of dimension $\dim S_x - 1$.
\end{lemma}

\begin{proof}
$S^f_x$ is the intersection of the unit sphere of $\rho$ 
with the subspace 
$S_x$. 
\end{proof}

\begin{example}
If for each $x \in M$,  
 all of the eigenvalues of $s_x$ are equal,
the bipartite spaces $(M,S,F^\pm)$ are Riemannian manifolds.
In this case $\bfr = \lambda \bfs$ where $\lambda : M \to [0,1)$.
\end{example}

\begin{example}
Let $\bfa: M \to TM$ be a vector field on $M$ 
so that $\rho(\bfa(x)) < 1$ for all $x \in M$.
Then the $\bfa$ manifolds of Example~\ref{abmfd}
are bipartite spaces~\cite{ak11}.
\label{examfd}
\end{example}

\begin{example}
\label{comp}
Let $\bfb: M \to TM$ be a vector field on $M$ 
so that $\rho(\bfb(x)) < 1$ for all $x \in M$.
Then the $\bfb$ manifolds of Example~\ref{abmfd}
are bipartite spaces~\cite{ak11}.

Note that the bipartite form for the $\bfa$ manifolds
can be written using the isomorphisms defined in Eq.~\eqref{isolr}
as $\bfs = \alpha \otimes \alpha$,
where $\alpha  = l_1(\bfa)$.
The bipartite form for the $\bfb$ manifolds
can be written 
as $\bfs =c_{1,1}^{n-1}({ \star\beta \otimes \star\beta})$,
where $\beta = l_1(\bfb)$
and $\star$ is the Hodge star operator on $\bigwedge V$,
and where the notation $c_{1,1}^{n-1} = c_{1,1} \circ \cdots c_{1,1}$
indicates contractions between all corresponding indices
on the two $\star \beta$ factors
except for the $n$th.
The complementarity of the $\bfa$ and $\bfb$ manifolds 
is reflected in the parallel between the direct-product expressions 
of the corresponding bipartite forms.
\end{example}

\begin{conjecture}
A bipartite space of the form $(M,S,F^+)$ 
is an almost Finsler manifold.
\label{conj1}
\end{conjecture}

\begin{conjecture}
A bipartite space of the form $(M,S,F^-)$ 
is an almost Finsler manifold if and only if for all $x \in M$, 
the slit $S_x = \ker(s_x)$ has dimension 0, $\dim M$, or $\dim M -1$.
\label{conj2}
\end{conjecture}

The rest of this subsection is devoted to evidence for these conjectures.

\begin{lemma}
\label{aminusmemma}
The $\bfa$ spaces $(M,S,F^\pm_{\bfa})$ are almost Finsler manifolds.
\end{lemma}

\begin{proof}
As described in Example \ref{vecdecompa},
 $(M,S,F^\pm_{\bfa})$ locally
matches one of the Randers spaces $(M,F_{\pm\bfa})$,
and the latter are almost Finsler manifolds, 
with the proof given in sections 11.1 and 11.2 of Ref.~\cite{bcs}. 

\end{proof}

Other special cases are known. 
Conjecture~\ref{conj1} has been proved in Ref.~\cite{krt12} for the case 
where $\bfs$ has one positive eigenvalue of arbitrary multiplicity.
This includes $\bfa$ space $(M,S,F^+_{\bfa})$
and $\bfb$ space $(M,S,F^+_{\bfb})$,
among other spaces.

In contrast,
the following theorem shows that some bipartite spaces $(M,S,F^-)$
are not almost Finsler manifolds.

\begin{theorem}
\label{bipminus}
Let $(M,S,F^-)$ be a 
bipartite partial Finsler manifold so that for each $x \in M$, $S_x \not = 0$.  Then $(M,S,F^-)$ is almost Finsler 
if and only if for all $x \in M$, the slit $S_x = \ker(s_x)$ 
has dimension  $\dim M -1$ or $\dim M$.
\end{theorem}
 
\begin{proof}
By focusing on a single tangent space, 
one sees that this is really a question about Minkowski spaces. 
Suppose $(V, \langle, \rangle)$ is an inner product space. 
Let $\bfs$ be a symmetric nonnegative $(0,2)$-tensor on $V$ 
and let $s: V \to V$ be the endomorphism 
so that $\bfs(\bfy,\bfy') = \langle s(\bfy),\bfy' \rangle $. 
Assume that the eigenvalues of $s$ are less than one 
and greater than or equal to zero and that there is a positive eigenvalue.
Let $\rho(\bfy) = \|\bfy\|$ and $\sigma(\bfy) = \sqrt{\bfs(\bfy,\bfy)}$.
Let $S = \ker (s)$.
We show that the partial Minkowski space $(V,S,F := \rho -\sigma)$
is almost Minkowski 
if and only if  $\dim S = \dim M$ or $\dim S = \dim M-1$. 

If $\dim S = \dim M$, 
then $F = \rho$, and $(V,S,F)$ is a Euclidean space.
If $\dim S = \dim M -1$, 
then $(V,S,F)$ is $\bfa$ space, 
where $\bfa = \lambda \bfe$ where $\bfe$ is a unit-norm eigenvector
of $s$ with eigenvalue $\lambda \in (0,1)$.
Indeed, $\sqrt{\bfs(\bfy,\bfy)} = |\langle \lambda \bfe, \bfy  \rangle |$ 
for all $\bfy$, since $\bfy = \mu \bfe + \bff$ 
where $\bff \in S^{\perp}$ by the spectral theorem.

Now assume that $\dim M - 2 \geq  \dim S \geq 1$; 
we wish to show that $(V,S,F)$ is not almost Minkowski.
By the spectral theorem, there is a $\rho$-orthonormal basis 
$\{\bfe_1, \dots, \bfe_n\}$ 
of eigenvectors of $s$ with eigenvalues $\lambda_1, \dots, \lambda_n$.
We choose the eigenvectors so that $\lambda_1 \not = 0, \lambda_2 \not = 0$, 
and $\lambda_n = 0$.
Use this basis to identify  $V$ with $\R^n$, and use local coordinates.
 
In these coordinates,
the fundamental tensor $g_{jk}$ 
for the bipartite Minkowski space $(\R^n,S,F)$
can be written as~\cite{krt12}
\begin{equation} \label{KR}
g_{jk} = {F \over \rho} \delta_{jk} 
+ \rho\sigma
\left(\frac{\rho_j}  {\rho} - {\sigma_j \over \sigma}\right)
\left({\rho_k \over \rho} - {\sigma_k \over \sigma}\right)
- {F \over \sigma} s_{jk} 
\end{equation}
where the subscripts $j$, $k$ on $\rho$ and $\sigma$ 
indicate taking partial derivatives.

Consider points on the indicatrix $I$ of the form 
$\bfy = (\varepsilon, 0, \dots, 0, 1 + \eta)$ with 
$\epsilon$ and $\eta$ small positive numbers.
Solving for $\eta$ in $1  = F(\bfy) =  \rho(\bfy) - \sigma(\bfy) 
 = \sqrt{\varepsilon^2 + (1 + \eta)^2} - \sqrt{\lambda_1}\varepsilon$ yields 
$$
\eta = \sqrt{1 + 2\sqrt{\lambda_1 }\varepsilon 
+ (\lambda_1 -1)\varepsilon^2 } -1 
= \sqrt{\lambda_1} \varepsilon + O(\varepsilon^2)
$$
by the Taylor series expansion of $\sqrt{1+x}$ about $0$. 

We use \eqref{KR} to show that $g_{22}(\bfy)$ 
is negative for small $\varepsilon$. Note 
\begin{align*}
\rho(\bfy) &= \sqrt{\varepsilon^2 + (1 + \eta)^2}
= \sqrt{1 + 2 \sqrt{\lambda_1} \varepsilon + O(\varepsilon^2) } 
= 1 + \sqrt{\lambda_1}\varepsilon + O(\varepsilon^2)\,,
\\
\sigma(\bfy) & = \sqrt{\lambda_1}\varepsilon\,,
\quad
\rho_2(\bfy)  = 0\,,
\quad
\sigma_2(\bfy)  = 0\,.
\end{align*}
Some computation then reveals that
\begin{align*}
g_{22}(\bfy) & = 
\frac{\rho(\bfy) -\sigma(\bfy)}{\rho(\bfy)} 
- \frac{\rho(\bfy) -\sigma(\bfy)}{\sigma(\bfy)} \lambda_{2}
= 1 - \lambda_2 -  \frac{\lambda_2}{\sqrt{\lambda_1}\varepsilon} 
+ O(\varepsilon),
\end{align*}
which is negative for sufficiently small $\varepsilon$.
Thus $\bfg(\bfy)(\bfe_2,\bfe_2) < 0$, 
so the fundamental tensor is not positive definite at $\bfy$.
\end{proof}

\begin{corollary}
\label{fminus}
The $\bfb$ space $(M,S,F^-_{\bfb})$ is an almost Finsler manifold
only when $\dim M =2$.
\end{corollary}

\begin{proof}
For $\dim M = 2$, 
$(M,S,F^-_{\bfb})$ equals $(M,S,F^-_{\bfa})$ 
where $\langle \bfa,\bfb\rangle = 0$ 
and $\rho(\bfa) = \rho(\bfb)$ (see also Ref.~\cite{ak11}).
Thus $(M,S,F^-_{\bfb})$ is almost Finsler.
For $\dim M > 2$,
$(M,S,F^-_{\bfb})$ is a bipartite space with $\dim M - 1$ 
equal eigenvalues~\cite{krt12},
and hence by Theorem \ref{bipminus} 
it is not an almost Finsler manifold.
\end{proof}

We conjectured above that a bipartite space $(M,S, F^-)$ 
is a Finsler manifold if the slit $S$ is the zero section $M$.
In this case $F^-$ is strictly convex, 
and thus the solid indicatrices are strictly convex, 
but to prove the conjecture one would need to show 
that $F^-$ is strongly convex in the sense of Ref.~\cite{bcs}.

\begin{remark}
A valid proof of Conjecture~\ref{conj1} is given in Ref.~\cite{krt12}
for the case $(M,S,F^+)$ where for each $x$, 
$s_x$ has only one positive eigenvalue of arbitrary multiplicity.
However,
the paper incorrectly asserts that under the same conditions
$(M,S,F^-)$ is also almost Finsler,
a result contradicted by Theorem \ref{bipminus}.
The error stems from an inequality 
in the text below Eq.~(6) of Ref.~\cite{krt12}.
\end{remark}

\subsection{Bipartite indicatrix unions}
\label{Bipartite indicatrix unions}

\begin{definition}
The {\em bipartite indicatrix union} $\widehat{I}$ is the union 
of the indicatrices $I^\pm$ of the almost Finsler manifolds 
with Finsler norms $F^\pm$
with the space $S^f$,
$\widehat{I} = I^+ \cup I^- \cup S^f$.
\end{definition}

The bipartite indicatrix union $\widehat{I}_{\bfa}$ 
of the indicatrices $I_{\bfa}^{\pm}$ of the $\bfa$ spaces 
with almost Finsler norms 
$F_{\bfa}^{\pm}$ given in Eq.~\eqref{amfd}
is isomorphic to the union 
$\widehat{I}_{\bfa} =  I_{\bfa} \cup I_{-\bfa}$
of the indicatrices $I_{\pm \bfa}$ 
of the Randers manifolds 
with Finsler norms 
$F_{\pm \bfa}$ given in Eq.~\eqref{rmfd}~\cite{ak11}.

For $n=2$ and $\langle \bfa, \bfb \rangle = 0$,
the bipartite indicatrix union $\widehat{I}_{\bfa}$ 
of the indicatrices $I_{\bfa}^{\pm}$ of the $\bfa$ spaces 
with almost Finsler norms 
$F_{\bfa}^{\pm}$ given in Eq.~\eqref{amfd}
equals
the bipartite indicatrix union $\widehat{I}_{\bfb}$
of the indicatrices $I_{\bfb}^\pm$ of the $\bfb$ spaces
with almost Finsler norms 
$F_{\bfb}^{\pm}(\bfy)$ given in Eq.~\eqref{bmfd}.
For $n>2$, 
$\widehat{I}_{\bfa}$ and $\widehat{I}_{\bfb}$
are distinct hypersurfaces, as shown by the computations below.

\begin{lemma}
For $n=2$,
the homology groups of $\widehat{I}_{\bfa}$ and $\widehat{I}_{\bfb}$
coincide:
$H_0 = \Z$, $H_1 = \Z^3$.
For $n>2$,
the homology groups of $\widehat{I}_{\bfa}$ are 
$H_0 = \Z$, $H_j = 0$, $j=1,\ldots,n-2$, $H_{n-1} = \Z^3$, 
and $H_j = 0$ for $j \geq n$. 
\end{lemma}

\begin{proof}
$\widehat{I}_{\bfa}$ is the union of two $(n-1)$-dimensional ellipsoids 
that intersect in an $(n-2)$-dimensional ellipsoid.
Hence $\widehat{I}_{\bfa}$ is homeomorphic 
to $S^{n-1} \cup_{S^{n-2}} S^{n-1}$.
The corresponding Mayer-Vietoris exact sequence gives the result.
\end{proof}

\begin{lemma} 
$$
H_*(\wh I_{\bfb}) = 
\begin{cases}
(\Z,\Z^3, 0,0,  \dots)\,, & n = 2\,, \\
(\Z,\Z,\Z^2, 0,0,  \dots)\,, & n = 3\,, \\
(\Z,0,\Z^2, \Z,  0, 0  \dots)\,, & n = 4\,, \\
(\Z,0,\Z,\Z, \Z,  0, 0  \dots)\,, & n = 5\,,
\end{cases}
$$
and if $n > 5$, then $H_j(\wh I_{\bfb})= \Z$ 
if and only if $j = 0,2,n-2,n-1$ and is zero otherwise.
\end{lemma}

\begin{proof}
For $n = 2$, $\widehat{I}_{\bfb} = \widehat{I}_{\bfa}$, so we assume $n>2$.

If $A$ is a nonempty closed subset of a space $X$ 
that is a deformation retract of a neighborhood in $X$, 
then $(X,A)$ is called a {\em good pair}
and there is an exact sequence~\cite{ah}
$$
\dots \to  H_2(X/A) \to H_1(A )\to H_1(X) \to  H_1(X/A) \to H_0(A) \to H_0(X).
$$

$\widehat{I}_{\bfb}$ is a surface of revolution 
obtained by rotating an $(n-2)$-dimensional ellipsoid 
about an intersecting line that does not contain 
the center of the ellipsoid (see Fig.~1).
Let $N = (0, \dots, 0, 1), S = (0, \dots, 0 , -1) \in S^{n-2}$.
Then there is continuous surjective map 
$f : S^{n-2} \times S^1 \to \widehat{I}_{\bfb}$ 
so that $f^{-1}(\{x\}) = \{N \} \times S^1$ 
and $f^{-1}(\{y\}) = \{S \} \times S^1$.
Note that $f$ is a quotient map since the domain is compact 
and the codomain is Hausdorff. 
Thus if $X_1 = (S^{n-2} \times S^1)/ (\{N \} \times S^1)$ 
then $\wh{I}_{\bfb} \cong X_1/ (\{S \} \times S^1)$.  

Recall the K\"unneth theorem $H_*(X \times S^1) = H_*(X) \oplus H_{*-1}(X)$.
Then in the exact sequence of the good pair 
$(S^{n-2} \times S^1, \{N \} \times S^1)$, 
the map $H_1(\{N \} \times S^1) \to H_1(S^{n-2} \times S^1)$ 
maps isomorphically to the summand $H_{1-1}(S^{n-2})$.
Thus the exact sequence of the good pair shows 
$$
 H_j(X_1)  = 
\begin{cases}
H_1(S^{n-2}), & j = 1, \\
H_j(S^{n-2}) \oplus H_{j-1}(S^{n-2}),  & j \not = 1.
\end{cases}
$$
The inclusion $ \{S \} \times S^1 \to X_1$ is null homotopic 
so the exact sequence of the good pair $(X_1, \{S \} \times S^1) )$ 
shows that
$$
H_j(\wh I_{\bfb}) = 
\begin{cases}
 H_2(X_1) \oplus \Z, & j = 2, \\
 H_j(X_1), & j \not = 2. 
\end{cases}
$$
Thus
$$
H_j(\wh I_{\bfb}) = 
\begin{cases}
H_1(S^{n-2}),    &  j = 1, \\
H_2(S^{n-2}) \oplus H_1(S^{n-2}) \oplus \Z, & j = 2, \\
H_j(S^{n-2}) \oplus H_{j-1}(S^{n-2}),  & j > 2.
\end{cases}
$$
The result follows.
\end{proof}

Note that in all cases the total Betti number 
(the sum of the ranks of the homology) 
of $\wh I_{\bfa}$ and $\wh I_{\bfb}$  is four. 
Perhaps this is further evidence for the complementary nature 
of $\bfa$ space and $\bfb$ space.

A partial classification of isomorphisms 
among various bipartite indicatrix unions
is given in Table 1 of \cite{krt12}. 

\begin{figure}[h!] 
\begin{center}
\includegraphics[width=2in]{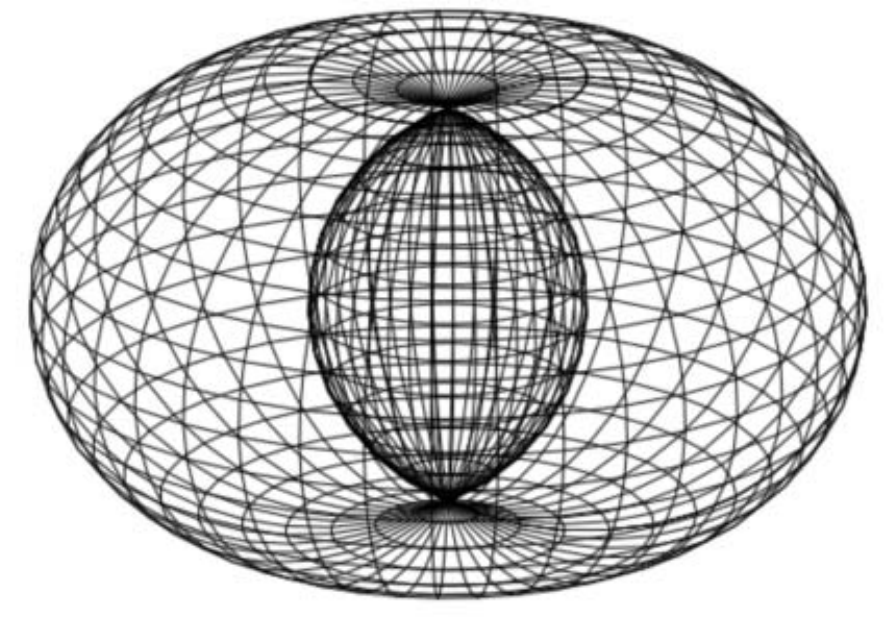}
\includegraphics[width=2.3in]{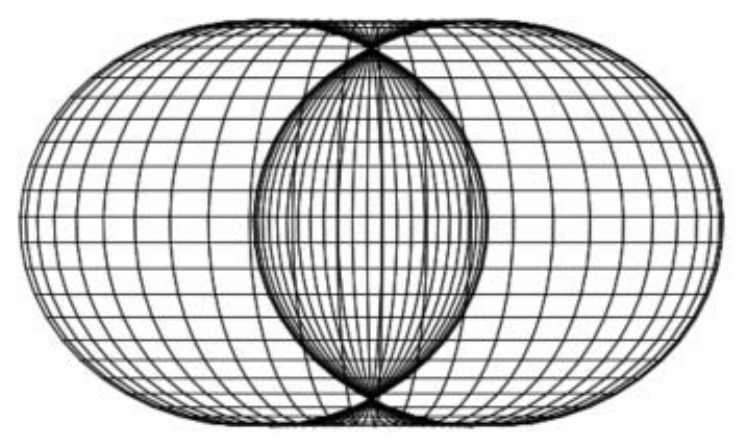}
\caption{The bipartite indicatrix union $\widehat{I}_{\bfb}$ for $n=3$.}
\end{center}
\label{fig1}
\end{figure}

\begin{lemma}
The bipartite indicatrix union $\widehat{I}_{\bfb}$
of the indicatrices $I_{\bfb}^\pm$ of the $\bfb$ spaces
is an $(n-1)$-dimensional spindle toroid $ST^{n-1}$.
\end{lemma}

The algebraic equations for the surface $ST^{n-1}$ 
can be expressed in convenient local coordinates
chosen to diagonalize and scale the Riemannian metric 
to the identity $r = I_n$
and to be oriented so that $\bfb = (0,\ldots,0,\pm b)$
with $0<b<1$.
Denoting these $n$ local coordinates by $(\bfy, y^n)\in \R^n$
so that $\rho^2 = \|\bfy\|^2 + (y^n)^2$,
the hypersurface $ST^{n-1}$ is the union of the two varieties given by 
$$
\frac{(\|\bfy\| \pm b B)^2}{B^2} + \frac{(y^n)^2}{B} = 1,
$$
where $B = 1/(1-b^2)$.

\begin{example}
For $n=3$, 
the bipartite indicatrix union $\widehat{I}_{\bfb}$
is the $2$-dimensional spindle toroid $ST^2$.
The surface $ST^2$ is displayed from two perspectives in Fig.~1. 
It is a degenerate Dupin cyclide~\cite{pcd}
formed from the union of the lemon $I_{\bfb}^+$ 
and the apple $I_{\bfb}^-$. 
\end{example}

\begin{example}
For $n=2$,
the bipartite indicatrix union $\widehat{I}_{\bfb}$
is the intersection of $ST^2$ with a plane through the fixed points $S^f$.
\end{example}

\section{Characteristic tensors}
\label{Characteristic tensors}

Suppose an almost Finsler manifold $(M,S,\ol F)$ is specified.
Then a set of tensor differential equations
for a Finsler norm $F$ with solution $F = \overline{F}$ 
can be obtained by taking $\bfy$ derivatives 
of the equation $F(\bfy) - \overline{F}(\bfy)=0$.
These differential equations therefore characterize
the almost Finsler manifold $(M,S,\overline{F})$.
Since certain $\bfy$ derivatives of $F$ 
are related to geometric quantities such as the fundamental tensor $\bfg$,
it is natural to seek a tensor combination of the differential equations
that can be expressed entirely in terms of geometric quantities.
Any such combination of geometric quantities
then can be interpreted as a characteristic tensor 
that vanishes for the prescribed almost Finsler manifold.
This interpretation in terms of characteristic tensors
has been achieved for Riemannian geometry
(Deicke's Theorem~\cite{ad})
and for Randers spaces 
(Matsumoto-H\=oj\=o Theorem~\cite{mh}).

Here,
we obtain a characteristic tensor that vanishes for bipartite geometries
that are almost Finsler manifolds
and consider several limiting cases,
including Riemann, Randers, $\bfa$, and $\bfb$ manifolds.

\subsection{Proof of Theorem \ref{bipthm}
}

Recall (see Eq.~\eqref{pb}) 
that for any partial Finsler manifold $(M, S, F)$ 
we have a Euclidean vector bundle 
$\pi_1: \pi^*|TM \to TM \backslash S$ with fiber $T_xM$
where $x = \pi(\bfy)$ and with metric $\bfg(\bfy)$.
We consider $(p,q)$-tensor fields with values in $\pi^*|TM$.
We work in local coordinates: 
for a point $x \in M$, 
we choose a neighborhood $U \subset M$ 
so that there are linearly independent sections 
$\{\bfe_1, \dots, \bfe_n\}$ of $TM|_U \to U$.
This gives linearly independent sections 
(denoted $\{\pi|^*\bfe_1, \dots, \pi|^*\bfe_n\}$) 
of $\pi_1$ restricted to $\pi_1^{-1}(U)$.
Then,
given real valued functions $y^1, \dots, y^n: U \to \R$, 
we obtain $\bfy(x') = \sum y^j\bfe_j(x')$ for $x' \in U$.

For the following discussion of bipartite spaces,
it is convenient to minimize clutter from signs
by abbreviating $F^\pm = F = \rho + \Delta$
with $\Delta = \pm \sigma$.
It is also convenient to work with $\bfy$ derivatives at each $T_xM$
taken in local coordinates $y^j$, $j = 1, \ldots, n$.
Derivatives are denoted by subscripts,
e.g., $F_j = \partial F/\partial y^j$. 

\begin{lemma}
A bipartite Finsler norm $F$ satisfies
the differential equation
\begin{equation}
(F^2)_{jkl} - 2 \Delta F_{jkl} - 2 F \Delta_{jkl}
-2 \sum_{(jkl)} \left(\Delta_j F_{kl} + \Delta_{jk} F_l \right) = 0, 
\label{diffeq}
\end{equation}
where the summation is over cyclic permutations of $jkl$.
\end{lemma}

This follows by squaring the expression $F-\Delta = \rho$,
which reveals that $F^2 - 2 \Delta F$ is a second-order polynomial in $\bfy$.
Taking three derivatives then generates the result. 

The Finsler norms $F$ and differences $\Delta = F-\rho$,
together with their derivatives,
are geometric quantities associated with bipartite spaces.
The following definitions are standard in the literature~\cite{bcs}.

\begin{definition}
The {\em Hilbert form} $\bfp$ 
is a $(0,1)$-tensor with local components 
$p_j = F_j$ and
the {\em angular metric} $\bfh$ 
is a symmetric $(0,2)$-tensor  with local components 
$h_{jk} = F F_{jk}$,
The {\em Cartan torsion} $\bfC$ from Example \ref{metric_and_Cartan}
is a totally symmetric $(0,3)$-tensor  with local components
$C_{jkl} = \tfrac 14 (F^2)_{jkl}$.
\label{def1}
\end{definition}

\begin{lemma}
The differential equation (\ref{diffeq})
can be expressed in terms of geometric quantities as 
\begin{equation}
2(F-\Delta) C_{jkl} 
+ \sum_{(jkl)}
\left(\frac{\Delta}{F} ~ p_j - \Delta_j\right)
\left(h_{kl} - \frac{F^2}{\Delta} \Delta_{kl}\right)
=0.
\label{geomeq}
\end{equation}
\label{geomlem}
\end{lemma}

Taking successive derivatives of $\Delta$ reveals that
the third derivative $\Delta_{jkl}$ is given by
$\Delta_{jkl} = -\frac{1}{\Delta}\sum_{(jkl)} \Delta_j \Delta_{kl}$.
Lemma \ref{geomlem} then follows from Eq.~(\ref{diffeq}) by substitution.

\begin{definition}
The {\em fundamental tensor} is the symmetric $(0,2)$-tensor
$\bfg= \bfh + \bfp \otimes \bfp$ mentioned in Example \ref{metric_and_Cartan}.
Viewing its local expression $g_{jk}$ as a matrix,
the matrix inverse $g^{jk}$ is the {\em inverse fundamental tensor}.
The {\em mean Cartan torsion} $\bfI$
is a 1-form with local components 
$I_j = g^{kl} C_{jkl}= \partial_{y^j}\ln(\sqrt{\det \bfg})$.
\label{def2}
\end{definition}

\begin{lemma}
The mean Cartan torsion $I_j$ for bipartite spaces is determined by
$$
2(F-\Delta) I_j 
+\kappa
\left(\frac{\Delta}{F} ~ p_j - \Delta_j\right)
+ \frac{2F^2}{\Delta} g^{kl}\Delta_k \Delta_{lj} = 0,
$$
where 
$\kappa = n+1 -\frac{F^2}{\Delta} g^{kl}\Delta_{kl}$.
\label{mctlem}
\end{lemma}

This result can be confirmed by contracting Eq.~(\ref{geomeq})
with $g^{jk}$
and applying homogeneity,
$F(\lambda\bfy) =\lambda F(\bfy)$,
using Euler's Theorem~\cite{bcs}.
Specifically,
since $y^j F_{jk} = 0$ by 1-homogeneity of $F$,
it follows from Definition \ref{def1}
that $y^j h_{jk} = 0$
and hence that $y^j = F g^{jk} p_k$,
so $g^{jk} p_j h_{kl} = 0$.
Also,
1-homogeneity of $F$ implies $y^j p_j = F$
and hence $g^{jk} p_j p_k = 1$.
Definition \ref{def2} then yields 
$g^{jk} h_{kj} = n-1$.
Since $\Delta$ is also 1-homogeneous,
it follows that
$g^{jk} p_j \Delta_{kl} = 0$.
Finally, 
$g^{jk} p_j \Delta_k = \Delta/F$,
which implies
$g^{jk} h_{kl} \Delta_j = \Delta_l - \Delta p_l/F$.
Substitution of these results into the contraction of Eq.~(\ref{geomeq})
verifies Lemma \ref{mctlem}.

Theorem \ref{bipthm} follows directly 
by combining Lemma \ref{geomlem} with Lemma \ref{mctlem}.
The characteristic tensor $S_{jkl}$ is the local representation
of a totally symmetric $(0,3)$-tensor $\bfS$,
the {\em bipartite tensor}.

\begin{example}
For Riemannian geometry
$\Delta = 0$,
and the characteristic tensor $\bfS$
reduces to the Cartan torsion,
$\bfS \vert_{\Delta = 0} = \bfC$,
consistent with Deicke's Theorem~\cite{ad}.
\end{example}

\begin{example}
For the Randers manifolds of Example \ref{abmfd},
adopt convenient local coordinates
chosen to scale the Riemannian metric to $r = I_n$
and to be oriented so that $\bfa = (0,\ldots,0,\pm a)$
with $0<a<1$.
Then $\Delta = \pm a y^n$ and so $\Delta_{jk}= 0$,
and the bipartite tensor $\bfS$ reduces to the {\em Matsumoto tensor} $\bfM$,
$\bfS \vert_{\Delta_{jk}=0} = \bfM$,
given by 
$\bfM = \bfC - \frac{1}{n+1} \sum_{(jkl)} \bfI \otimes \bfh$.
This is consistent with the Matsumoto-H\=oj\=o Theorem
for Finsler manifolds~\cite{mh}.
\label{exrndel}
\end{example}

\begin{example}
For the almost Finsler $\bfa$ manifolds of Example \ref{abmfd}
and in the same convenient local coordinates as Example \ref{exrndel},
$\Delta = \pm \vert a \vert \vert y^n \vert$.
Away from the slit $S$,
we have $\Delta_{jk}= 0$ 
and the bipartite tensor $\bfS$ again reduces to the Matsumoto tensor $\bfM$,
$\bfS \vert_{\Delta_{jk}=0} = \bfM$.
We thus see that
the introduction of almost Finsler manifolds
allows additional manifolds solving the Matsumoto condition $\bfM =0$. 
\end{example}

\begin{remark}
By construction,
the $\bfg$ trace of $\bfS$ vanishes,
$g^{jk} S_{jkl} = 0$.
This encompasses the vanishing 
of the $\bfg$ trace of the Matsumoto tensor $\bfM$,
$g^{jk} M_{jkl} = 0$.
\label{strace}
\end{remark}

\subsection{Proof of Theorem \ref{bthm}
}

Next,
we determine the characteristic tensor for the $\bfb$ spaces,
which are special bipartite spaces.
The derivation takes advantage of projection operators 
specific to the $\bfb$ spaces
to evaluate key parts of the expression in Theorem \ref{bipthm}.

The decomposition of $\bfy$ in Eq.~\eqref{vecdecomp} 
can be understood as a decomposition 
of the identity map $\bfI:~ V \to V$ 
into projectors parallel and perpendicular to $\bfb$,
$\bfI = \bfP_\parallel + \bfP_\perp$,
where for nonzero $\bfb$ the parallel projector can be written as
$\bfP_\parallel = \bfb \otimes l_1(\bfb)/\|\bfb\|^2$.
In local coordinates,
we denote the components of the perpendicular projector $\bfP_\perp$
as $P^j{}_k = \delta^j{}_k - b^j b_k/\|\bfb\|^2$,
and we write $P_{jk}= r_{jl} P^l{}_k$,
$P^{jk}= P^j{}_l r^{lk}$,
and $(Py)_j = P_{jk} y^k$.

\begin{lemma}
The perpendicular projector $P^j{}_k$ satisfies 
$P^j{}_k P^k{}_l = P^j{}_l$.
Its trace is 
${\rm tr}({\bfP_\perp}) = P^j{}_k \delta^k{}_j = P_{jk} r^{jk} = n-1$,
and it annihilates $\bfb$, $P^j{}_k b^k = 0$.
\end{lemma}

Many interesting properties of the $\bfb$ spaces
can be traced to features of the perpendicular projector $\bfP_\perp$.

\begin{lemma}
The $\bfb$ spaces are bipartite spaces with bipartite form given as 
$s_{jk} = \|\bfb\|^2 (\bfP_\perp)_{jk} = \|\bfb\|^2 P_{jk}$.
\end{lemma}

The $\bfa$ spaces are bipartite spaces with bipartite form given as
$s_{jk} 
= \|\bfa\|^2 (\bfP_\parallel)_{jk}
= \|\bfa\|^2 (\bfI - \bfP_\perp)_{jk}$.
This reflects the complementarity of the $\bfa$ and $\bfb$ spaces
described in Example \ref{comp}.

For the $\bfb$ spaces,
taking the $y^j$-derivative of $\Delta^2=y^j s_{jk} y^k$
shows that $\Delta \Delta_j = \|\bfb\|^2 (Py)_j$.
This implies directly that $\Delta_j$ annihilates $\bfb$,
$\Delta_j b^k= 0$,
that the $r$-norm of $\Delta_j$ is given by
$\|{\bf \Delta}\|^2 = \Delta_j \Delta^j = \|\bfb\|^2$,
and that $\rho^j \Delta_j = \Delta/\rho$.
The latter expression also follows from 
the 1-homogeneity of $F$ and of $\rho$,
which together imply the 1-homogeneity of $\Delta$.

Substitution of the result 
$\Delta \Delta_j = \|\bfb\|^2 (Py)_j$
into the $y^k$-derivative of $\Delta \Delta_j$ gives
$\Delta \Delta_{jk} = \|\bfb\|^2 P_{jk} 
- \frac{\|\bfb\|^4}{\Delta^2} (Py)_j(Py)_k$. 
With these expressions in hand, and introducing 
${\widetilde\Delta}^j{}_k = \frac{\Delta}{\|\bfb\|^2} \Delta^j{}_k$,
calculation yields the following lemma.

\begin{lemma}
\label{delproj}
The quantity ${\widetilde\Delta}^j{}_k$ is a projector,
${\widetilde\Delta}^j{}_k {\widetilde\Delta}^k{}_l = {\widetilde\Delta}^j{}_l$.
Its trace is 
${\widetilde\Delta}^j{}_k \delta^k{}_j = n-2$,
and it annihilates $\bfb$ and ${\bf\Delta}$: 
${\widetilde\Delta}^j{}_k b^k = 0$
and ${\widetilde\Delta}^j{}_k \Delta^k = 0$.
\end{lemma}

An application of Euler's Theorem~\cite{bcs}
using the 1-homogeneity of $\rho$ implies
$y^j = r^{jk}\rho\rho_k$.
A second application of Euler's Theorem 
using the 1-homogeneity of $\Delta$ 
then reveals that 
$\Delta_{jk} \rho^j = 0$.
The projector ${\widetilde\Delta}^j{}_k$ 
therefore also annihilates $\rho^k$
and hence by Lemma \ref{delproj} annihilates $F^j$ as well:
${\widetilde\Delta}^j{}_k \rho^k= 0$
and 
${\widetilde\Delta}^j{}_k F^k= 0$.

The above elegant features of $\Delta_j$ 
and the projector ${\widetilde\Delta}^j{}_k$ 
stem from the properties of the perpendicular projector $P^j{}_k$.
Moreover,
calculation shows that $\Delta_j$  and ${\widetilde\Delta}^j{}_k$ 
are eigenvectors of $P^j{}_k$ in the following sense.

\begin{lemma}
\label{projproj}
The actions of the perpendicular projector $\bfP_\perp$ 
on $\Delta_j$ and on the projector ${\widetilde\Delta}^j{}_k$ are 
$P^j{}_k \Delta^k = \Delta^j$
and $P^j{}_k \Delta^k{}_l = \Delta^j{}_l$.
\end{lemma}

To obtain the characteristic tensor for the $\bfb$ spaces,
we need to evaluate the quantities
$\kappa$ and $g^{kl}\Delta_k\Delta_{lj}$
appearing in the statement of Theorem \ref{bipthm}.
For this purpose,
the explicit form of the inverse fundamental tensor is required.

\begin{lemma}
\label{invmetric}
The inverse fundamental tensor $g^{jk}$ for the $\bfb$ spaces is~\cite{ak11}
$$
g^{jk} =
\frac {\rho} {F} 
\left(
r^{jk}
+ \frac{(\bfb\cdot\bfy)^2 \rho}
{\Delta^2(\Delta + \|\bfb\|^2 \rho)} \lambda^j\lambda^k
- \frac {\|\bfb\|^2 \rho}{(\Delta + \|\bfb\|^2 \rho)} P^{jk}
\right),
$$
where
$\lambda_j = \frac{\bfb\cdot\bfy}{F} \rho_j - b_j$,
$\lambda^j = r^{jk}\lambda_k$.
\end{lemma}

To determine $\kappa$,
we need an explicit expression for  $g^{jk}\Delta_{jk}$
for the $\bfb$ spaces.
Using the form of Lemma \ref{invmetric} for the inverse metric
and the results in Lemmas \ref{delproj} and \ref{projproj}
yields the following result.

\begin{lemma}
\label{quantity1}
The value $\kappa_b$ of $\kappa$ for the $\bfb$ spaces is
$$\kappa_b = n+1 - (n-2)
\frac {F(F-\Delta)\|{\bf\Delta}\|^2}
{\Delta\big(\Delta + (F-\Delta) \|{\bf\Delta}\|^2\big)}.$$
\end{lemma}

Note that for $n=2$ this reduces to the usual factor
in the expression for the Matsumoto tensor~\cite{mh}
for the Randers and $\bfa$ spaces, $\kappa_b = n+1 =  \kappa_a$. 

\begin{lemma}
\label{quantity2}
For the $\bfb$ spaces, $g^{kl}\Delta_k\Delta_{lj} = 0$.
\end{lemma}

This follows from Lemmas \ref{delproj} and \ref{invmetric}.
The result depends on the properties of the projectors, 
and hence it is a special feature of the $\bfb$ spaces.

Combining Lemmas \ref{quantity1} and \ref{quantity2}
with Theorem \ref{bipthm}
directly yields the desired characteristic tensor for the $\bfb$ spaces,
thereby proving Theorem \ref{bthm}.

The $\bfg$ trace of $\bfB$ vanishes,
$g^{jk} B_{jkl} = 0$,
in concordance with Remark~\ref{strace}.

A characteristic tensor $\bfT$ with the structure
$\bfT \sim \bfC - \sum_{jkl} \bfI \times (*)$
is said to be {\em quasi $C$ reducible}~\cite{mh}.
The form for the characteristic tensor $B_{jkl}$ given by Theorem \ref{bthm} 
reveals that $\bfb$ spaces are quasi $C$ reducible,
where the factor $(*)$ 
takes the special form of a shift of the angular metric.

\bmhead{Acknowledgments}

This work is supported in part 
by the U.S.\ Department of Energy 
under grant {DE}-SC0010120, the Simons Foundation 
under grants 713226 and 00013966,
and by the Indiana University Center for Spacetime Symmetries.

\end{document}